\documentclass[12pt,a4paper,reqno]{amsart}
\usepackage[margin=2.6cm]{geometry}
\usepackage{setspace}
\usepackage{graphics, amsmath,amssymb, enumitem, comment, url}
\usepackage{graphicx}
\usepackage{epsfig}

\newif\ifcolorcomments
\newcommand{\allowcomments}[4]{
\newcommand{#1}[1]{\ifdraft{\ifcolorcomments{\textcolor{#4}{##1 --#3}}\else{\textsl{ ##1 \ --#3}}\fi}\else{}\fi}
}

\allowcomments{\commumtaz}{MH}{Mumtaz}{green}
\allowcomments{\comjohannes}{JS}{Johannes}{blue}
\allowcomments{\comdavid}{DS}{DS}{magenta}

\colorcommentstrue
\usepackage{amssymb}
\usepackage{amsmath,amsthm}

\usepackage{colortbl}

\usepackage{amssymb,amsmath}
\usepackage{geometry}    
\usepackage{mathrsfs}

\newtheorem{theorem}{Theorem}[section]
\newtheorem{lemma}[theorem]{Lemma}
\newtheorem{proposition}[theorem]{Proposition}
\newtheorem{corollary}[theorem]{Corollary}
\theoremstyle{definition}
\newtheorem{definition}[theorem]{Definition}
\newtheorem{remark}{Remark}
\newtheorem{example}[theorem]{Example}
\newtheorem{conjecture}{Conjecture}
\newtheorem{problem}{Problem}

\newcommand{\DD}{\mathcal D}

\newcommand{\HH}{\mathcal H}

\newcommand{\MM}{\mathcal M}

\newcommand{\bfP}{\mathbf P}
\newcommand{\Q}{\mathbb Q}

\newcommand{\R}{\mathbb R}

\newcommand{\Z}{\mathbb Z}

\newcommand{\mbf}{\mathbf}
\newcommand{\0}{\mbf 0}% Moved since \00 is not a valid macro

\newcommand{\ee}{\mbf e}

\newcommand{\qq}{\mbf q}
\newcommand{\rr}{\mbf r}

\newcommand{\vv}{\mbf v}

\newcommand{\xx}{\mbf x}
\newcommand{\yy}{\mbf y}
\newcommand{\zz}{\mbf z}

\renewcommand{\text}{\textup}

\newcommand{\NPC}[1]{\ignorespaces}

\newif\ifdraft\drafttrue

\IfFileExists{marvosym.sty}{
\RequirePackage{marvosym} 
}{

}

\IfFileExists{wasysym.sty}{
\RequirePackage{wasysym}
}{

}

\newcommand{\WWW}{W_{n}^{\theta}(\Psi)}

\newcommand{\DDD}{\DD_n^{\theta}(\Psi)}

\newtheorem*{GBSP1}{Generalised Baker-Schmidt Problem for Hausdorff Measure: dual setting}

\begin{document}
%\title[A note on dual approximation on manifolds]{A note on dual approximation on manifolds
%of dimension at least three }

\title[The GBSP for dual approximation]{The Baker-Schmidt problem for dual approximation and some classes of functions}
\author{Mumtaz ~Hussain}
\address{Mumtaz Hussain, Department of Mathematical and Physical Sciences, La Trobe University,  Bendigo 3552, Australia. }
\email{m.hussain@latrobe.edu.au}

\author{Johannes Schleischitz}
\address{Middle East Technical University, Northern Cyprus Campus, Kalkanli, G\"uzelyurt}
\email{johannes@metu.edu.tr;  jschleischitz@outlook.com}

\begin{abstract} 
The Generalised Baker-Schmidt Problem (1970) concerns the Hausdorff $f$-measure of the set of $\Psi$-approximable points on a nondegenerate manifold. 
%There are two variants of this problem, concerning simultaneous and %dual approximation. 
We refine and extend our previous work  [Int. Math. Res. Not. IMRN 2021, no. 12, 8845--8867] in which we settled the problem 
(for dual approximation) for hypersurfaces. We verify the GBSP for certain classes of nondegenerate submanifolds of 
codimension greater than $1$. Concretely, for codimension two or three, we provide examples of manifolds
where the dependent variables can be chosen as quadratic forms.
Our method requires the manifold
to have even dimension at least the minimum of four and half the dimension of the ambient space. We conjecture that these restrictions on the dimension of the manifold are sufficient to provide similar examples in general.
% not quite true for n=3 only partial result !

%
% Beresnevich-Dickinson-Velani (in 2006, for the homogeneous setting) and Badziahin-Beresnevich-Velani (in 2013,  for the inhomogeneous setting) proved the divergence part of this problem for dual approximation on arbitrary nondegenerate manifolds. The corresponding convergence counterpart represents a major challenging open question and the progress thus far has only been attained over  hypersurfaces in our previous paper [Int. Math. Res. Not. IMRN 2021, no. 12, 8845--8867]  for inhomogeneous approximations with a non-monotonic multivariable approximating function. In this paper we refine our previous paper and extend those results to any nondegenerate manifolds with the exception of one or two dimensional manifolds. 
\end{abstract}

\maketitle

\section{Dual Diophantine approximation on manifolds}
Let $n\geq 1$ be a fixed integer, $\qq:=(q_1,\ldots, q_n)\in \Z^n$,  and $\xx=(x_1,\dots,x_n)\in\R^n$. 
Let $\Psi:\Z^n\to[0, \infty)$ be a \emph{multivariable approximating function}, that is,  $\Psi$ has the property that
$\Psi(\qq) \rightarrow~0 \text{ as } \|\qq\|:=\max(|q_1|, \ldots, |q_n|) \rightarrow~\infty.$ Let $\theta$ be an arbitrary real number\footnote{
We remark that, in fact, our results
still apply for any sufficiently 
smooth function $\theta(\xx)$, as in~\cite{HSS}, see Remark~\ref{fehlt} below.}. Consider the set
\begin{equation*}
  \DDD:=\left\{\xx\in\R^n:\begin{array}{l}
  |\qq\cdot\xx+p+\theta|<\Psi(\qq)  \  \text{for} \  {\rm i.m.} \ (p, \qq)\in\Z\times \Z^{n}
                           \end{array}
\right\},
\end{equation*}
where  `i.m.' stands for `infinitely many'.  A vector $\xx \in \R^n$ will be called \emph{$(\Psi, \theta)$-approximable} if it lies in the set $\DDD$. When $\theta= 0$, the problem reduces to the \emph{homogeneous} setting.
We are interested in the `size' of the set $\DDD$ with respect to the $f$-dimensional Hausdorff measure $\HH^f$ for some dimension function $f$. By a dimension function $f$ we mean an increasing continuous function $f:\mathbb{R}\to \mathbb{R}$ with $f(0)=0$. %In the case where the dimension function is of the form $f(r):=r^s$ for some $s < k$, $\HH^f$ is simply denoted as $\HH^s$.
%  
%The most modern result in regards to determining the size of the set $\DDD$ is the following statement due to the contributions of many authors but most importantly due to the works of Schmidt \cite{Schmidt8} and Beresnevich \& Velani \cite{BeresnevichVelani4}.
%
%
%\begin{thmsch}
%Let $\Psi$ be a multivariable approximating function. Let $f$ be a dimension function such that $r^{-n}f(r)\to\infty$ as $r\to 0.$ Assume that $r\mapsto r^{-n}f(r)$ is decreasing and $r\mapsto r^{1-n}f(r)$ is increasing. Fix $\theta\in\R$.  Then
%\begin{equation*}
%  \HH^f( \DDD)=\left\{\begin{array}{cl}
% 0 &  {\rm if } \quad\sum\limits_{\qq\in\Z^n\setminus \{\0\}}\|\qq\|^n\Psi(\qq)^{1-n}f\left(\frac{\Psi(\qq)}{\|\qq\|}\right)< \infty.\\[3ex]
% \infty &  {\rm if } \quad \sum\limits_{\qq\in\Z^n\setminus \{\0\}}\|\qq\|^n\Psi(\qq)^{1-n}f\left(\frac{\Psi(\qq)}{\|\qq\|}\right)=\infty \text{ and $\Psi$ is decreasing or $n\geq 2$}.
%                                     \end{array}\right.
%\end{equation*}
%\end{thmsch}Theorem SBV is very significant as it encompasses many classical landmark results such as Khintchine's theorem, Jarn\'ik's theorem, and Groshev's theorem.
%
%\subsection{Diophantine approximation on manifolds: dual settings}

Diophantine approximation on manifolds concerns the study of approximation properties of points in $\R^n$ which are functionally related or in other words restricted to a submanifold $\MM$ of $\R^n$. To estimate the size of sets of points $\xx\in\R^n$ which lie on a $k$-dimensional, nondegenerate\footnote{In this context `nondegenerate' means suitably curved, see \cite{Beresnevich3,KleinbockMargulis2} for precise formulations}, analytic submanifold $\MM\subseteq \R^n$ is an intricate and challenging problem. The fundamental aim is to estimate the size of the set $\MM\cap\DDD$  in terms of Lebesgue measure, Hausdorff measure,  and Hausdorff dimension. When asking such questions it is natural to phrase them in terms of a suitable measure supported on the manifold, since when $k<n$ the $n$-dimensional Lebesgue measure is zero irrespective of the approximating functions. For this reason,  results in the dependent Lebesgue theory (for example, Khintchine-Groshev type theorems for manifolds) are posed in terms of the $k$-dimensional 
Lebesgue measure (equivalent to the Hausdorff measure) on $\MM$.

  %Throughout, the induced Lesgue measure on a given manifold will be denoted by $\lambda$.

In full generality, a complete Hausdorff measure treatment for manifolds $\MM$  represents a deep open problem referred to as the Generalised Baker-Schmidt Problem (GBSP) inspired by the pioneering work of Baker and Schmidt \cite{BakerSchmidt}.   There are two variants of this problem, concerning simultaneous and dual approximation. In this paper we are concerned with the dual approximation only.   Ideally one would want to solve the following problem in full generality.

\begin{GBSP1} Let $\MM$ be a nondegenerate submanifold of $\R^n$ with $\dim\MM=k$ and $n \geq 2$. Let $\Psi$ be a multivariable approximating function.  Let $f$ be a dimension function such that $r^{-k}f(r)\to\infty$ as $r\to 0.$ Assume that $r\mapsto r^{-k}f(r)$ is decreasing and $r\mapsto r^{1-k}f(r)$ is increasing. Prove that 
% rewriting ! is->be; 2nd Then removed

\begin{equation*}
  \HH^f( \DDD\cap\MM)=\left\{\begin{array}{cl}
 0 &  {\rm if } \quad\sum\limits_{\qq\in\Z^n\setminus \{\0\}}\|\qq\|^k\Psi(\qq)^{1-k }f\left(\frac{\Psi(\qq)}{\|\qq\|}\right)< \infty;\\[3ex]
 \infty &  {\rm if } \quad \sum\limits_{\qq\in\Z^n\setminus \{\0\}}\|\qq\|^k\Psi(\qq)^{1-k}f\left(\frac{\Psi(\qq)}{\|\qq\|}\right)=\infty.
                                     \end{array}\right.
\end{equation*}
% rewriting ! setminus 0 added
\end{GBSP1}
Note that $\HH^f$ is proportional to the standard Lebesgue measure when $f(r)=r^n$. In fact, the GBSP is stated in the most idealistic format and solving it in this form is extremely challenging. The main difficulties lie  in the convergence case and therein constructing a suitable nice cover for the set $ \DDD\cap\MM$. 
Recently (2021), the authors with David Simmons settled the GBSP for hypersurfaces for both homogeneous and inhomogeneous settings with non-monotonic multivariable approximating functions \cite{HSS}. We also proved several results for the one-dimensional manifolds such as nondegenerate planar curves or for Veronese curves in \cite{HSS2} under some regularity conditions on the dimension function.  In this paper we refine our framework set out in \cite{HSS} and extend those results to certain classes of nondegenerate manifolds of higher codimension.

We refer the reader to \cite[Subsection 1.1]{HSS} for a description of historical progression towards the GBSP.  Specifically, we refer the reader to \cite{BHH, BakerSchmidt, Huang, Huang2, Hussain, HSS2, KleinbockMargulis2}. For the recent developments of the GBSP for the simultaneous approximation, we refer the reader to \cite[Remark 1.1]{HSS} or \cite{Huang}.
%analogue of this problem and developments in that direction

\subsection{Setup and main result}  \label{sandr}

 We first state some regularity conditions on the manifold $\MM$ and
the dimension function $f$. Let $n\ge 2$ be an integer and $\MM$ be a manifold of codimension $l\ge 1$ in $\R^{n}$. 
Assume more precisely that $\MM$ is a surface of a $C^2$-map $g:U\to \R^{l}$, where $U\subset\R^{n-l}$ 
is a connected bounded open set. We denote this as $\MM = \Gamma(g)$.
By default we will assume that vectors are line vectors,
and use superscript $t$ for their transpose.
 Let $\MM$ be parametrised by
\begin{equation} \label{eq:via}
\MM= \{ (x_{1},\ldots, x_{n-l}, g_{1}(\xx),\ldots,
g_{l}(\xx)): \;
\xx=(x_{1},\ldots,x_{n-l})\in U\subseteq \R^{n-l}\}.
\end{equation}
Our conditions read as follows:

\begin{itemize}
\item[(I)] Let $f$ be a dimension function satisfying
 \begin{equation}
\label{fhypothesis}
f(xy) \lesssim x^s f(y), \qquad \text{ for all } y < 1 < x
\end{equation}
for some $s < 2(n-l-1)$.
\item[(II)] 
The real symmetric $(n-l)\times (n-l)$ matrix $\Lambda=\Lambda(g, \boldsymbol{s},\xx)$ with entries
\[
\Lambda_{j,i}=  \sum_{u=1}^{l} s_u\cdot \frac{\partial^2g_u}{\partial x_i\partial x_j}(\xx), \qquad 1\le i,j\le n-l,
\]
for $g$ as above is regular for any choice of real $\boldsymbol{s}=(s_1, s_2,\ldots, s_l)\ne \0$ and all $\xx\in U\setminus S_{\MM}$ outside a set $S_{\MM}$ of $f$-measure $0$.
\end{itemize}

While (II) turns out to be rather complicated,
we point out that (I) holds as soon as the manifold has dimension at least three, as in~\cite{HSS}. Section~\ref{iandii} of this paper is reserved for a more detailed discussion on the conditions, followed
by examples of manifolds satisfying the hypotheses in Section~\ref{examples}. 
Our new criterion for the convergence part of the GBSP, proved in Section~\ref{proof}, reads as follows.

%we have that 
%\[
%P(\yy,\zz(\xx))\ne 0, \qquad \xx\in U, \; \yy\in \R^{l-1},\; \Vert %\yy\Vert \le 1.  
%\]
%Hereby $\zz(\xx)$ is some explicit vector of length $\binom{l+1}{2}$ %derived in some way from partial derivatives of $g$ evaluated at %$\xx\in U$.
%\comjohannes{At the moment the proof requires at most one solution
%but the method should extend}  \commumtaz{This may be regarded as the new definition of singularity??, i.e. whenever this property is not satisfied then we call such a manifold... } 

\begin{theorem}\label{thm1} Let $\Psi$ be a multivariable approximating function. Let $f$ be a dimension function satisfying \text{(I)} and let $g: \R^{n-l}\to \R^{l}$ be a $C^2$-function satisfying \text{(II)}. Then if $\MM$ is given
	via $g$ by \eqref{eq:via}, then
	\begin{equation*}\label{fw}
	\HH^f(\DDD\cap\MM)=0
	\end{equation*}
	if the series
	\begin{equation*}\label{eqcon}
	\sum\limits_{\qq\in\Z^n\setminus \{\0\}}\|\qq\|^{n-l}\Psi(\qq)^{l+1-n}f\left(\frac{\Psi(\qq)}{\|\qq\|}\right)
	\end{equation*}
	converges.
	% rewriting ! CLAIM CORRECTED exponents were false n-> n-1,
	% and formulation added
	
\end{theorem}

\begin{remark}  \label{fehlt}
	It is possible to formulate a variant of Theorem~\ref{thm1}
	with a functional smooth inhomogenity $\theta(\xx)$ as in~\cite{HSS}.
\end{remark}

\subsection{Corollaries (from a combination with the divergence results)
}\label{sectionfibering} 

In this section, we detail some of the corollaries of our theorem along with some of the consequences. We begin by summarising the notation used.
\smallskip

\noindent{\bf Notation.} In the case where the dimension function is of the form $f(r):=r^s$ for some $s < k$, $\HH^f$ is simply denoted as $\HH^s$. On occasions we will consider functions of the form $\Psi(\qq)=\psi(\|\qq\|)$, and in this case we use $W_n^\theta(\psi)$ as a shorthand for $W_n^\theta(\Psi)$. The function $\psi:\R_{>0} \to \R_{>0}$ is called a \emph{single-variable approximating function}. 
$B_n(\xx, r)$ denotes a ball centred at the point $\xx\in\mathbb{R}^n$ of radius $r$.
For real quantities $A,B$ and a parameter $t$, we write $A \lesssim_t B$ if $A \leq c(t) B$ for a constant $c(t) > 0$ that depends on $t$ only (while $A$ and $B$ may depend on other parameters). 
We write  $A\asymp_{t} B$ if $A\lesssim_{t} B\lesssim_{t} A$.
If the constant $c>0$ depends only on parameters that are constant throughout a proof, we simply write $A\lesssim B$ and $B\asymp A$.

\smallskip

The divergence part of the GBSP was proved by Badziahin-Beresnevich-Velani \cite{BBV} for the $s$-dimensional Hausdorff measure\footnote{It is to be noted that the techniques used in proving \cite[Theorem 2]{BBV}  also work for the  $f$-dimensional Hausdorff measure situation. } and for multivariable approximating functions satisfying a certain property $\bfP$ for any nondegenerate manifolds. Following the terminology of \cite{BBV}, we say that an approximating function $\Psi$ satisfies property $\bfP$ if it is of the form
$\Psi(\qq)=\psi(\|\qq\|_{\vv})$ for a monotonically decreasing
function $\psi: \mathbb{R}_{>0}\to \mathbb{R}_{>0}$ 
(single-variable approximation function), 
$\vv=(v_{1},\ldots,v_n)$ with $v_i > 0$ and $\sum_{1\leq i\leq n} v_{i}=n$,
and $\| \cdot \|_{\vv}$ defined as
the quasi-norm $\|\qq\|_{\vv}= \max_i \vert q_{i}\vert^{1/v_{i}}$.
When combined with Theorem~\ref{thm1} we obtain the following implication on the GBSP problem for nondegenerate (see~\cite{BBV}) manifolds for dimensions not less than $3$.

\begin{corollary}\label{HSS:cor2}
	Let $\Psi$
	be a decreasing %\comjohannes{ word added }
	multivariable approximating
	function satisfying property $\bfP$. Let $f$ be a dimension function satisfying \text{(I)} and let $g$ be a $C^2$-function satisfying \text{(II)}. Let $\MM$
	be a nondegenerate manifold in $\mathbb{R}^{n}$ 
	of dimension $n-l$, given
	via $g$ by \eqref{eq:via}. Then
	
	\begin{equation*}
	\HH^f( \DDD\cap \MM)= \begin{cases}
	0 &  {\rm if } \quad \sum\limits_{\qq\in\Z^n\setminus \{\0\}}\|\qq\|^{n-l}\Psi(\qq)^{l+1-n}f\left(\frac{\Psi(\qq)}{\|\qq\|}\right)< \infty;\\[3ex]
	\infty &  {\rm if } \quad \sum\limits_{\qq\in\Z^n\setminus \{\0\}}\|\qq\|^{n-l}\Psi(\qq)^{l+1-n}f\left(\frac{\Psi(\qq)}{\|\qq\|}\right)=\infty.
	\end{cases}
	\end{equation*}
	%where $\MM$ is the graph of $g$.
\end{corollary}

We emphasise again that only the divergence case, treated in~\cite{BBV}, 
assumes monotonicity on the approximating function $\Psi$.  
%As usual, 
%we write $\HH^s$ for $\HH^f$ when $f(x) = x^s$, $s>0$.  

\begin{corollary}\label{maincor}
	Let $\MM$ and $\Psi$ be as in Corollary~\ref{HSS:cor2}
	and let $s$ be a real number satisfying $s<2(n-l-1)$. Then
	\begin{equation*}
	\HH^s( \DDD\cap \MM)= \begin{cases}
	0 &  {\rm if } \quad \sum\limits_{\qq\in\Z^n\setminus \{\0\}}\|\qq\|\left(\frac{\Psi(\qq)}{\|\qq\|}\right)^{s+l+1-n}< \infty;\\[3ex]
	\infty &  {\rm if } \quad \sum\limits_{\qq\in\Z^n\setminus \{\0\}}\|\qq\|\left(\frac{\Psi(\qq)}{\|\qq\|}\right)^{s+l+1-n}=\infty,
	\end{cases}.
	\end{equation*}
\end{corollary}

%\comjohannes{ I am confused, where does $s>n-2$ come from? Nothing  like that in HSS. If should be there then it should be explained why }

For $l=1, n=2$ where $\MM$ is a planar curve, while the corollary and its proof are formally valid, the involved parameter range for $s$ is empty. We refer the reader to \cite{HSS2} for the GBSP type results on nondegenerate curves such as Veronese curves, i.e. sets of the form $\{(x,x^{2},\ldots,x^n):x\in\R\}$. We remark that
in view of Corollary~\ref{maincor},
Conjecture~\ref{C1} below would imply the GBSP for a large class of manifolds and any $\Psi$ with
property $\mathbf{P}$.

One of the consequences of Corollary \ref{maincor} is the following Hausdorff dimension result. Let $\tau_{\Psi}$ be the lower order at infinity of $1/\Psi$, that is,
\[
\tau_\Psi:=\liminf_{t\to\infty} 
\frac{\log(1/\Psi(t))}{\log t},
\ \  \text{where} \ \  \Psi(t)= \inf_{\xx\in\R^n: \|\xx\|\le t} \Psi(\xx),
\]
%\comjohannes{ i think indeed we need $\le t$ in range }
and we may assume $\tau_\Psi\ge n> 0$. 
Then from the definition of Hausdorff measure, Theorem \ref{thm1} implies that for any approximating function $\Psi$ with lower order at infinity $\tau_{\Psi}$ and for $\MM$ as above
(that is, with property (II)) of dimension $\dim\MM\geq 3$, we have
\[
\dim_\HH (\WWW\cap\MM)\leq \dim \MM-1+\frac{n+1}{\tau_{\Psi}+1}.
\]
%\comjohannes{ do we need to further restrict $\tau$ for this? i think %its fine though }
This Hausdorff dimension result was previously only known for either the planar curve \cite{BHH, Hussain, Huang}, Veronese curve \cite{Bernik}, or for the hypersurface \cite{HSS}.

\medskip
\noindent{\bf Acknowledgments.} The research of Mumtaz Hussain is supported by the Australian Research Council Discovery Project (200100994).

\section{On conditions (I) and (II)}  \label{iandii}

\subsection{On Condition (I)} \label{reh}
Note that in the statement of the GBSP, the standard condition on the dimension
function $f$ that $f(q)q^{-(n-l)}$ (recall $n-l=\dim \MM$) is decreasing is assumed. With this condition in hand, the condition (I) is satisfied as soon as 
\[
n-l\geq 3. 
\]
From the aforementioned decay property we see
that $f(xy)\le x^{n-l}f(y)$ for $y<1<x$. Since $n-l<2(n-l-1)$
as soon as $n-l\ge 3$, indeed we infer that \eqref{fhypothesis} holds
in the non-empty parameter range $s\in [n-l, 2(n-l-1))$. On the other hand,
condition (I) excludes all curves (one dimensional manifolds),
as well as all two dimensional manifolds, for
many interesting functions $f$.

\subsection{On Condition (II)}  \label{condii}
Condition (II) is more delicate. It replaces and generalises the non-vanishing of the
determinant of the Hessian $\nabla^2 g\in \mathbb{R}^{(n-1)\times (n-1)}$ condition within $U$ from~\cite{HSS} when $l=1$. 
Fixing $\xx\in U$,
the determinant of $\Lambda$ in (II) becomes a multivariate, homogeneous polynomial $P$ of total degree $n-l$ in the variables $s_1,\ldots,s_l$. Its coefficients are functions in the
$\ell\cdot \binom{n-l+1}{2}$ second order partial derivatives of $g$
(using symmetry of $\Lambda$), evaluated at $\xx\in U$. 
In other words, for (II) we need that some form of degree $n-l$ in $l$ variables is positive (or negative) definite at any point in $U$.
This definiteness problem of forms
seems related to Hilbert's XVII-th problem, whose positive answer in particular implies that $P$ above 
must have a presentation
as a sum of squares. It is evident that the complexity of the problem increases fast
as $l$ and $n$ grow. 

%By homogenity, 
%we may restrict to $\Vert \boldsymbol{s}\Vert:=\max |s_i|=1$
%to lower the degree to $n-l-1$. 
A priori, it is not clear if, generally,  $g$
satisfying hypothesis (II) exists when $l\ge 2$. However,
for certain pairs $(l,n)$ we provide examples below.
It is evident from the form of $\Lambda$ that, provided that examples exists at all,  we can choose the coordinate functions $g_j$ as quadratic forms, defined on the entire space $U=\R^{n-l}$. Moreover, we can perturb any such example by adding any functions with small enough second order derivatives by absolute value uniformly on $\R^{n-l}$ to the $g_j$. For simplicity of presentation, we introduce the following notion.

\begin{definition}
	Call a pair of integers $(l,n)$ with $n>l\ge 1$ a {\em good pair}
	if (II) holds for some $g:\R^{n-l}\to \R^l$ as in the introduction.
\end{definition}

By the above observation, the induced $g$ of any good pair can be 
taken a quadratic form defined globally. It is obvious that
any pair $(1,n)$ is good. For larger $l$,
the following general going-up and going-down 
properties for good pairs are straightforward.

\begin{proposition}  \label{prop}
	If $(l,n)$ is a good pair, then so is 
	\begin{itemize}
		\item[\rm (i)] the pair
		$(l-t,n-t)$ for any integer $0\le t\le l-1$.
		\item[\rm (ii)] the pair $(l,\tilde{n})$ with $\tilde{n}=n+t(n-l)$ 
		%\[
		%\tilde{g}=(g_1,\ldots,g_{l},g_1,\ldots,g_{l},\ldots,g_1,\ldots,g_{%l})
		%\]
		%	where $g=(g_1,\ldots,g_{l})$ is repeated $t$ times,
		for any integer $t\ge 0$.
	\end{itemize}
\end{proposition}

Claim (i) can be seen by specialisation of $t$ variables, 
for example via putting $s_1~=~\cdots=~s_{t+1}$.
The proof of (ii) will become apparent from the examples in Section~\ref{examples}. On the other hand, 
it is in general unclear if $(l,n)$ being a good pair 
will
imply the same for $(l,\tilde{n})$ with a general larger $\tilde{n}>n$, even if $\tilde{n}-l$ is even (see Obstruction~1 below).

On the other hand,
condition (II) has some natural limitations, as captured in the following obstructions.

\noindent $\textbf{Obstruction 1:}$  %\label{rem}
	If $l\ge 2$ and $n-l$ is odd, then $(l,n)$ is not a good pair, that is, the condition (II) cannot hold. 
	
	While a short proof seems to follow directly from the positive
	answer to Hilbert's XVII'th problem, we want to explicitly
	explain how this special case can be handled.
	First assume $l=2$. Indeed, if 
	either the coefficient of $x_1^{n-l}$ or $x_2^{n-l}$ of the polynomial $P(x_1,x_2)$ defined above vanishes, then we may take $x_1$ arbitrary
	and $x_2=0$ or vice versa, hence we get non-trivial solutions for $\det \Lambda=P(x_1,x_2)=0$.
	If otherwise both coefficients are not $0$, then by choosing any
	$x_2\ne 0$ we get a single-variable polynomial in $x_1$ of odd degree with non-zero constant term,  again inducing a non-trivial solution for $\det \Lambda=P(x_1,x_2)=0$.
	Finally if $l>2$, by specialising $l-2$ variables $x_3=\cdots=x_l=0$, 
	we get a polynomial $Q(x_1,x_2)$ in two variables. 
	Regardless if $Q\equiv 0$ is the constant
	$0$ polynomial or not, by the observations for $l=2$ above, 
	we may choose at least one of $x_1, x_2$ (almost)
	arbitrary for a solution of $\det \Lambda=0$. 
	Thus we again find a non-trivial
	solution for $\det \Lambda=0$ in $s_1,\ldots,s_l$. 
%\end{remark}

\noindent $\textbf{Obstruction 2:}$
	We require $n\ge 2l$ for $(l,n)$ being a good pair. 
	
	Otherwise if $n-l<l$ then
	we can annihilate a line of $\Lambda$ on a linear subspace
	of dimension at least $1$ of $\boldsymbol{s}\in\R^l$ 
	obtained from intersecting $n-l$
	hyperplanes of $\R^l$, regardless of the choice of $g, \xx$. 

Obstruction~2 means that the manifold must have at least 
half the dimension of the ambient space. We believe that these are the only obstructions.
In view of Theorem~\ref{thm1} and
the observations in Section~\ref{reh} we therefore go on to
state the following conjecture.

\begin{conjecture}  \label{C1}
	The following claims hold
	\begin{itemize}	
		\item[(i)] 
		If $l\ge 2$ and $n\ge 2l$ and $n-l$ is even, then $(l,n)$ is 
		a good pair.
		\item[(ii)] If $l\ge 2$ and $n\ge 2l$ and $n-l\ge 4$ is even, then there
		exists $g$ with coordinate functions $g_j$ quadratic forms such that (I), (II) holds on $\R^{n-l}$,
		hence the convergence part of the GBSP holds for the induced manifolds.
	\end{itemize}
\end{conjecture}

Clearly it would suffice to verify (i), claim (ii) is just stated for completeness.
We will verify the conjecture for $l=2$ in Section~\ref{examples}.
Moreover, for $l=3$ we establish the partial result 
that $n-l$ being a positive multiple of $4$ is sufficient. The simplest cases where Conjecture~\ref{C1} remains open 
are $l=3, n=9$ and $l=4, n=8$.

\section{Examples}  \label{examples}

\subsection{Special case $l=2$}

We start with an example to illustrate condition (II) in the
case $n=4, l=2$. Unfortunately, since
$\dim\MM=2$, condition (I) does not hold in this context, see Section~\ref{reh}.

\begin{example}  \label{Ex1}
Let $l=2$, $n=4$, so that $\MM$ is a two-dimensional manifold
with codimension two. Using multilinearity of the
determinant and after some calculations, we see that the polynomial 
representing $\det \Lambda$ becomes
\[
P(s_1,s_2)= 
s_1^2 A_1(g,\xx) +   s_2^2 A_2(g,\xx) + s_1s_2 A_3(g,\xx)
\]
where
\[
A_1(g,\xx)= \frac{\partial^2g_1}{\partial^2x_1}(\xx)\cdot  \frac{\partial ^2g_1}{\partial^2x_2}(\xx)-(\frac{\partial^2g_1}{\partial x_1\partial x_2}(\xx))^2,  %\det \begin{bmatrix} (d^2g_1/dx_1dx_1(\xx)) & %(d^2g_j/dx_idx_k(\xx)) \\ (d^2g_j/dx_idx_k(\xx)) & (d^2g_j/dx_idx_k(\xx))
%\end{bmatrix}, \qquad A_2= \det , A_3=\det 
\]
\[
A_2(g,\xx)= \frac{\partial^2g_2}{\partial ^2x_1}(\xx)\cdot  \frac{\partial^2g_2}{\partial^2x_2}(\xx)-(\frac{\partial^2g_2}{\partial x_1\partial x_2}(\xx))^2,  %\det \begin{bmatrix} (d^2g_1/dx_1dx_1(\xx)) & %(d^2g_j/dx_idx_k(\xx)) \\ (d^2g_j/dx_idx_k(\xx)) & (d^2g_j/dx_idx_k(\xx))
%\end{bmatrix}, \qquad A_2= \det , A_3=\det 
\]
and
\[
A_3(g,\xx)= \frac{\partial ^2g_1}{\partial ^2x_1}(\xx)\cdot  \frac{\partial ^2g_2}{\partial ^2x_2}(\xx)-2\frac{\partial ^2g_1}{\partial x_1\partial x_2}(\xx)
\cdot \frac{\partial ^2g_2}{\partial x_1\partial x_2}(\xx)
+ \frac{\partial ^2g_2}{\partial ^2x_1}(\xx)\cdot  \frac{\partial ^2g_1}{\partial ^2x_2}(\xx).
%\det \begin{bmatrix} (d^2g_1/dx_1dx_1(\xx)) & %(d^2g_j/dx_idx_k(\xx)) \\ (d^2g_j/dx_idx_k(\xx)) & (d^2g_j/dx_idx_k(\xx))
%\end{bmatrix}, \qquad A_2= \det , A_3=\det 
\]
%By symmetry and $\Vert\boldsymbol{s}\Vert=1$, we may assume $s_2=\pm 1$
%and by changing sign of $s_1$ in fact $s_2=1$, if we only
%restrict to $s_1\in\R$. 
By the criterion of minors to test definiteness of a quadratic
form with respect to the corresponding symmetric matrix with rows $(A_1(g,\xx), A_3(g,\xx)/2)$
and $(A_3(g,\xx)/2, A_2(g,\xx))$,
then the criterion 
of condition (II) that $P$ is positive (or negative) definite 
becomes
%\[
%A_1(g,\xx)s_1^2+ A_3(g,\xx)s_1 + A_2(g,\xx) \ne 0,\qquad \xx\in U,\; %s_1\in \R.
%\]
%Thus we need that the roots are complex, that is
\begin{equation} \label{eq:inin}
A_3(g,\xx)^2 < 4A_1(g,\xx)A_2(g,\xx).
\end{equation}
A sufficient criterion for second order derivatives is given by
\[
\frac{\partial ^2g_1}{\partial ^2x_1}(\xx)\cdot  \frac{\partial ^2g_1}{\partial ^2x_2}(\xx) =
\frac{\partial ^2g_2}{\partial ^2x_1}(\xx)\cdot  \frac{\partial ^2g_2}{\partial ^2x_2}(\xx)<0,\qquad 
\frac{\partial ^2g_1}{\partial x_1\partial x_2}(\xx)\ne  \frac{\partial ^2g_2}{\partial x_1\partial x_2}(\xx)
\]
as then we may write inequality
\eqref{eq:inin} equivalently as
\[
4\frac{\partial ^2g_2}{\partial ^2x_1}(\xx)\cdot  \frac{\partial ^2g_2}{\partial ^2x_2}(\xx)\cdot 
\left(\frac{\partial ^2g_1}{\partial x_1\partial x_2}(\xx)- \frac{\partial ^2g_2}{\partial x_1\partial x_2}(\xx)\right)^2 <0.
\]
Specialising further, if for example, we can take at some $\xx\in U$ the derivates
\begin{align*}
&\frac{\partial ^2g_1}{\partial ^2x_1}(\xx)=\frac{\partial ^2g_2}{\partial ^2x_1}(\xx) =2, \\ & 
\frac{\partial ^2g_1}{\partial ^2x_2}(\xx)= \frac{\partial ^2g_2}{\partial ^2x_2}(\xx)=-2, \\ & 
\frac{\partial ^2g_1}{\partial x_1\partial x_2}(\xx)\ne  \frac{\partial ^2g_2}{\partial x_1\partial x_2}(\xx),
\end{align*}
then it will be true in some neighbourhood of $\xx$.
For example when $n=4, l=2$ and $\delta\ne 0$, then
\[
g(\xx)=g(x,y)= ( x^2-y^2+\delta xy, x^2-y^2 )
\]
satisfies this for all $(x,y)$ in $\R^2$. The arising 
parametrised manifolds becomes
\begin{equation} \label{eq:doet}
\MM_{\delta}= \{ (x,y, x^2-y^2+\delta xy, x^2-y^2): x,y\in \R  \},  \ {\rm where} \ \delta\ne 0.
\end{equation}
We remark that $\delta\ne 0$ is necessary for the convergence part of GBSP as otherwise the manifold
lies in the rational subspace of $\R^4$ defined by $x_3=x_4$.
\end{example}

%\begin{example}[Degenerate examples for $l=3, n=5$]
%From any manifold as in the previous Example~\ref{Ex1} we can derive %an example for
%a manifold of the same codimension two in space dimension increased
%by one. Let for instance
%\[
%\MM_{\delta}= \{ (x,y, x^2-y^2+\delta xy, x^2-y^2,0): x,y\in \R  %\}\subseteq \R^5, \qquad \delta\ne 0.
%\]
%More generally, we may take the last coordinate $ax+by$ for 
%real numbers $a,b$.
%Then the determinant condition $\det \Lambda\ne 0$ 
%is the same as before, but with $\xx=(x,y,z)\in \R^5$.
%So we get the additional line $\{ (x,y,z): x=y=0 \}$ coming
%from $s_3=\pm 1$ and $s_1=s_2=0$. 
%\end{example}

We now present some examples satisfying both (I) and (II). 
Keeping $l=2$, we can extend the previous example to general even $n\ge 6$ by essentially building Cartesian products. A possible class of manifolds derived from this method that verifies Conjecture~\ref{C1} for $l=2$ is captured in the following example. 

\begin{example} \label{e2}
	Let $l=2, n-l=2t$ for $t\ge 2$. Then for any $\delta_1,\ldots,\delta_t\ne 0$ the manifold
		\[
	\MM=\left\{ \left(x_1,y_1,\ldots,x_t,y_t, \sum_{u=1}^{t} 
	x_u^2-y_u^2+\delta_u x_uy_u, \sum_{u=1}^{t} 
	x_u^2-y_u^2\right):\;\; x_i,y_i\in \R  \right\},
	\]
		satisfies (I) and (II). Hence the GBSP holds for any 
		decreasing $\Psi$ with property $\mathbf{P}$.
\end{example}

Indeed, the critical determinant $\det(\Lambda)$ 
in Example~\ref{e2} decomposes as a product of
$t$ determinants as in Example~\ref{Ex1}, which we found all 
to be non-zero,
so (II) holds. Since $n-l\ge 4$ condition (I) holds too. 
Clearly Example~\ref{e2} can be generalised in terms of parameter
ranges for the coefficients of the quadratic forms.
By Obstruction~1, the condition that $n-l$ is even
is necessary for (II) (unless $n-l=1$, the hypersurface case).
A similar argument proves Proposition~\ref{prop}.

%By the same argument,
%generally we have the following lifting property from %Proposition~\ref{pro}.%
%
%\begin{example} \label{EE}
%	If (II) holds for some pair $(l,n)$ and some $g:\R^{n-l}\to \R^l$, %then it holds as well for $(l,\tilde{n})$ with $\tilde{n}=n+t(n-l)$ %and some modified $\tilde{g}:\R^{\tilde{n}-l}\to\R^l$,
	%\[
	%\tilde{g}=(g_1,\ldots,g_{l},g_1,\ldots,g_{l},\ldots,g_1,\ldots,g_{%l})
	%\]
%	where $g=(g_1,\ldots,g_{l})$ is repeated $t$ times,
	%for any integer $t\ge 0$.
%\end{example}

%In general, in the next easiest case $l=2, n=6$ there are already
%$l\cdot \binom{n-l+1}{2}=20$ variables coming from second order %derivatives of $g$ occur,
%that have to satisfy $5$ positivity conditions on the minors of a %$5\times 5$-matrix to give a positive definite matrix.  

\subsection{The case $l>2$}

Now let us consider $l>2$. Together with the restrictions 
\[
n-l\equiv 0\bmod 2,\qquad n-l\ge 3,\qquad n\ge 2l,
\]
from Section~\ref{reh} and Obstructions~1 and 2, the easiest example is $l=3, n=7$.

The following general criterion for (II), or good pairs, 
involving ``definite determinants''
essentially comes from specialising certain variables (second order derivatives of $g$).

\begin{lemma} \label{laura}
	Let $n> l\ge 1$ be integers. 
Assume there exists a symmetric $(n-l)\times (n-l)$ matrix $M_{l,n}$ 
with the entries 
\[
M_{l,n}(i,j)\in \{ 0, \pm z_1 , \ldots, \pm z_l \}, \qquad 1\le i,j\le n-l,
\]
for 
formal variables $z_v$, $1\le v\le l$, so that
$\det M_{l,n} \ne 0$ for any choice of real $z_v$ not all $0$.
Then $(l,n)$ is a good pair.
\end{lemma}

Hence the existence of $M_{l,n}$ as in the lemma implies that
there exist $(n-l)$-dimensional submanifolds of $\R^n$, 
defined as in \eqref{eq:via} via $g=(g_1,\ldots,g_{l})$ with $g_j(\xx)$ real quadratic forms, that satisfy (II) on $U=\R^{n-l}$.	
Again, the hypothesis of Lemma~\ref{laura} forces $n\ge 2l$
and $n-l$ to be even.

\begin{remark}
	More generally, instead of $M_{l,n}(i,j)=\pm z_v$ or $0$, 
	we may take the matrix entries arbitrary
	linear combinations of the $z_v$, for the same conclusion.
	This generalised condition is in fact equivalent to (II).
	The proof is essentially the same as below. 
\end{remark}

%The claim becomes stronger the larger
%$l$ is compared to $n$, as we can also decrease $l$ by choosing
%certain variables equal.

\begin{proof}[Proof of Lemma~\ref{laura}] 
	Choose any set of $l$ linearly independent (over $\R$) 
	vectors $L_v=(a_{1,v},\ldots,a_{l,v})$, $1\le v\le l$, in $\R^l$ and define the
	linear forms in $l$ variables $s_1,\ldots,s_l$
	\[
	L_{v}\cdot \boldsymbol{s}= a_{1,v}s_1+\cdots + a_{l,v}s_l, \quad 1\leq v \le l.
	\] 
	Given $M_{l,n}$ as in the lemma, we construct a matrix
	$M_{l,n}^{\prime}$ with entries in formal variables $s_1,\ldots,s_l$ as follows:
	We identify $z_v$ with $L_v\cdot \boldsymbol{s}$, that is,  if for
	a pair $(i,j)$ the index $v=v(i,j)$ is so that
	$M_{l,n}(i,j)=\pm z_v$,
	then we choose the entry at position $(i,j)$
	of $M_{l,n}^{\prime}$  
	equal to $M_{l,n}^{\prime}(i,j)=L_v\cdot \boldsymbol{s}$. Else if $M_{l,n}(i,j)=0$ then keep  the value $M_{l,n}^{\prime}(i,j)=0$. Then $M_{l,n}^{\prime}$ is well-defined and depends on $s_1,\ldots,s_l$.
	By linear independence of the $L_v$ and the hypothesis of the lemma, 
	for any non-trivial choice of real numbers $s_1, \ldots,s_l$, the 
	matrix $M_{l,n}^{\prime}$ has non-zero determinant. 
	We finally
	notice that there is a one-to-one correspondence between
	any such collection of coefficient vectors $L_{v}$ 
	and a collection of quadratic forms
	$g_1,\ldots, g_l$
	via identifying for $v=v(i,j)$ as above
	$a_{u,v}=\partial^2 g_u/\partial x_i\partial x_j$ for $1\le u\le l$. In other words,
		$a_{u,v}x_ix_j$ if $i\ne j$ and $(a_{u,v}/2) x_i^2$ if $i=j$
	is the term containing $x_ix_j$ resp. $x_i^2$ in the quadratic form $g_u(\xx)=\sum_{i,j} a_{u,v}x_ix_j$. Hence (II) holds for the globally defined function $g=(g_1,\ldots,g_l)$, so $(l,n)$ is a good pair.
%	
%	
%	For now, fix $\xx\in \R^{n-l}$ and write %$\Lambda_{i,j}=\Lambda_{i,j}(\xx)$. 
%We identify 
%\[
%M_{l,n}(i,j)= \Lambda_{i,j}= \sum_{u=1}^{l} s_u\cdot  %\frac{\partial^2 g_u(\xx)}{\partial x_i \partial x_j}.
%\] 
%Then since $M_{l,n}$ satisfies
%the hypothesis of the lemma, there is some subset
%\[
%\mathcal{S}\subseteq \{1,2,\ldots,(n-l) \}\times \{1,2,\ldots,(n-l) \}
%\]
%of indices $i,j$ with $i\ge j$ 
%of cardinality $|\mathcal{S}|=l$, so that
%all the $l$ linear forms $\Lambda_{i,j}=L_{i,j}\cdot \boldsymbol{s}$ %with the coefficient vectors
%\[
%L_{i,j}=(\partial^2 g_1/\partial x_i \partial x_j \; , \ldots, \;
%\partial^2 g_l/\partial x_i \partial x_j), \qquad (i,j)\in \mathcal{S}
%\]
%vanish only if $s_1=\cdots=s_l=0$. 
%It then suffices to choose the $l$ linear forms $L_{i,j}$ linearly %independent (over $\R$),
%to conclude that all $s_u$ vanish, as required in (II).
%Clearly there is a one-to-one correspondence between
%any such collection of coefficient vectors $L_{i,j}$ 
%and a collection of quadratic forms
%$g_1,\ldots, g_l$, for which the claim holds globally.
\end{proof}

In the proof, we may just let $L_v$ the canonical base vectors in $\R^l$ for $1\le v\le l$ so that 
essentially $z_v$ equals $s_v$, however the proof contains more information.

Hence the problem (II) can be relaxed 
to finding suitable matrices $M_{l,n}$
as in the lemma.
For $l=2, n=4$ we can take the matrix
\[
M_{2,4}=\begin{pmatrix}
z_1 & z_2 \\
z_2 & -z_1
\end{pmatrix}
\]
which leads to Example~\ref{Ex1}. More generally, 
for $l=2$ and $n=2v$ with
$v\ge 2$ an even integer, we may take a matrix consisting of $v-1$ copies of $2\times 2$ ``diagonal'' blocks as above and zeros elsewhere, similar to Example~\ref{e2} and Proposition~\ref{prop}.

For $l=3, n=7$ we can take 
\begin{equation} \label{eq:37}
M_{3,7}=\begin{pmatrix}
z_1 & z_2 & z_3 & -z_{3} \\
z_2 & -z_1 & z_3 & z_3   \\
z_3 & z_3 & z_1 & z_2  \\
-z_3 & z_3 & z_2 & -z_1
\end{pmatrix}
\end{equation}
which has determinant $\det(M_{3,7})= (z_1^2+z_2^2)^2 + 4z_3^4$.
Our construction leads to the following example.

% NO !!!!!!!!!!!!!
%More generally, for arbitrary $l\ge 2$ we may choose $n=3l-2$ via the 
%symmetric matrix
%\[
%M_l=\begin{pmatrix}
%z_1 & z_2 & z_3 & -z_{3} & z_4 & -z_4 & \cdots & z_l & -z_l \\
%z_2 & -z_1 & z_3 & z_3 & z_4 & z_4 & \cdots & z_l & z_l  \\
%z_3 & z_3 & z_1 & z_2 & 0 & 0 & \cdots & \cdots & 0 \\
%-z_3 & z_3 & z_2 & -z_1 & 0 & 0 & \cdots & \cdots & 0 \\
%z_4 & z_4 & 0 & 0 & z_1 & z_2 & 0  & \cdots & 0 \\
%-z_4 & z_4 & 0 & 0 & z_2 & -z_1 & 0 & \cdots  & 0  \\
%\ddots & \ddots & \ddots & \ddots & \ddots & \ddots & \ddots & \ddots %& \ddots  \\
%z_l & z_l & 0 & \cdots & \cdots & \cdots & 0 & z_1 & z_2 & \\
%-z_l & z_l & 0 & \cdots & \cdots & \cdots & 0 & z_2 & -z_1 
%\end{pmatrix}
%\]
%where we have $2\times 2$-blocks of matrix $(z_1, z_2; z_2, -z_1)$ 
%in the ``diagonal'' and similar blocks $(z_u, z_{u}; -z_{u}, z_u)$ in %the first two lines and columns. It has determinant
%\[
%\det(M_l)= (-1)^{l}\cdot (z_1^2+z_2^2)^2\cdot 
%\left((z_1^2+z_2^2)^2+ 4\sum_{j=3}^l z_j^4 + 8 \sum_{3\le i<j\le l} %z_i^2 z_j^2 \right)
%\] 
%which indeed is easily checked only to vanish if $z_1=z_2\cdots=z_l=0$.

\begin{example}
Let $l=3, n=7$.
Taking the linear forms $L_{v}$ from the proof 
of Lemma~\ref{laura}
the canonical base vectors $(1,0,0), (0,1,0), (0,0,1)$ 
of $\R^3$ and inserting for the $z_v$ 
in $M_{3,7}$ from \eqref{eq:37}, 
leads to $g=(g_1,g_2,g_3)$ with quadratic form 
entries $g_u(\xx)=g_u(x_1,x_2,x_3,x_4)$ given by
\begin{align*}
&g_1(\xx)=\frac{x_1^2-x_2^2+x_3^2-x_4^2}{2}, \\ 
&g_2(\xx)= x_1x_2+x_3x_4,\\
 &g_3(\xx)= x_1x_3+x_2x_3+x_2x_4-x_1x_4.
\end{align*}
So, the manifold becomes
\[
\mathcal{M}= \{ (x_1,x_2,x_3,x_4, g_1(\xx), g_2(\xx), g_3(\xx): x_i\in\R \},
\]
which satisfies (I) and (II),  and thus GBSP for
any decreasing $\Psi$ with property $\mathbf{P}$. 
\end{example}

For $l=3$ we may take any $n\in\{7,11,15,19,\ldots\}$
again by repeating this $4\times 4$-block matrix $M_{3,7}$ 
along the ``diagonal''.
It is unclear if $l=3$ and $n\in\{ 9, 13, 17, \ldots \}$ can be achieved. It would suffice to verify this for $n=9$ to
infer the claim for all $n$ in the list by considering 
matrices decomposing into two types of diagonal blocks,
$M_{3,7}$ and the vacant $M_{3,9}$. This
would confirm Conjecture~\ref{C1} for $l=3$ as well. Unfortunately, we 
are unable to find a suitable matrix $M_{3,9}$.
%\comjohannes{ would be great to find it (if exists)!!!! }

The above discussion on Lemma~\ref{laura} motivates the following problem implicitly stated within Conjecture~\ref{C1}.

\begin{problem}
	Given $l\ge 3$, what is the minimum $n$ so that a matrix $M_{l,n}$
	as in Lemma~\ref{laura} exists? Equivalently, given even $n-l$,
	what is the largest $l$ for which the hypothesis holds for some matrix.
\end{problem}

It is unclear if such $n$ exists at all if $l\ge 4$. 
As remarked in Section~\ref{condii},
in all examples of Section~\ref{examples}, we can manipulate the  
manifold by adding any functions with uniformly (in absolute value) small enough second order derivatives to the $l$ functionally dependent variables. In particular, for any analytic functions 
defined on a neighbourhood of $\0\in\R^{n-l}$ with quadratic terms as in our examples above, GBSP holds upon possibly shrinking the neighbourhood.
%Generalizing the two previous examples, we get

%\begin{theorem}
%	Let $l,n$ be even and $n-l\ge 4$ and $n\ge 2l$.
	% equivalently:  Let $n-l\ge 4$ be even an $n\ge 2l$ be even. 
%	Then there exist 
%	quadratic forms defining $(n-l)$-dimensional manifolds in $\R^n$
%	such that GBSP holds for any $f$. 
%\end{theorem}

\section{Proof of Theorem \ref{thm1}} \label{proof}

Let us first clarify some notation. Recall
that by default all vectors
are row vectors, and we indicate with superscipt $t$ the transpose
(thus a column vector). In particular, for $g$ as above we consider
\[
g(\xx)=(g_{1}(\xx),\ldots,
g_{l}(\xx))\in \R^{l}, \qquad \xx\in U,
\]
as a row vector. Consequently the gradient $\nabla h$ of any multivariate
function $h$ with scalar output will be viewed as a line vector 
as well, and accordingly the columns of its Hesse matrix $\nabla^2 h$ 
are formed by the partial
derivatives of the corresponding fixed component function of the gradient $\nabla h$.
%Furthermore, write $\nabla g^{t}$ for the $(n-l)\times l$ matrix
%with $i$-th partial derivative vector in $\R^l$ of the transpose
%of $g$ (thus a column vector) in the $i$-th column ($1\leq i\le n-l$). %Denote by $\nabla^2 g_u^{t}\in \mathbb{R}^{(n-l)\times (n-l)}$ the %Hessians
%of the (transposed) coordinate functions $g_u^{t}$, $1\le u\le l$, %where
%again the $i$-th column consists of the partial derivative vector %$\partial g_u/\partial x_i$, $1\le i\le n-l$.

The proof is a refinement and extension of arguments presented in our previous paper \cite{HSS}. Here we only detail some modifications and other necessary details.

For any $\qq\in \Z^n$ and $p\in\Z$, analogously to~\cite{HSS} we define
\[
S(\qq,p) = S_{\Psi,\theta}(\qq,p) = \{\xx\in K : |\qq\cdot (\xx,g(\xx))^{t} - p - \theta| < \Psi(\qq)\},
\]
% rewriting ! wording changed
where $K$ is a compact subset of $U$. However, notice that 
our $g(\xx)$ is 
a vector here. Write 
$\qq = \tilde{q}\cdot(\rr,\boldsymbol{s})$ 
% rewriting  D-> n; D not needed?
for 
\[
\tilde{q}=\max_{n-l+1\leq i\leq n}\vert q_{i}\vert
\]
and some $\rr\in\Q^{n-l}$, $\boldsymbol{s}\in\Q^{l}$
% rewriting ! n -> n-1 : R -> Q
and let $a = a(\qq)= (p+\theta)/\tilde{q}$ and 
$\rho = \Psi(\qq)/|\tilde{q}|$ if $\tilde{q}\ne 0$ (assume this for now).
Then $\|\boldsymbol{s}\| =1$, and $\tilde{q} \le \|\qq\|$ by definition.
For any $\qq\in \Z^n$ and $p\in\Z$, we will bound the size of this set
which is equivalently given as
\[
S(\qq,p) = \{\xx\in K : |\rr\cdot\xx^{t} + \boldsymbol{s}\cdot g(\xx)^{t} - a| < \rho\}.
\]
Similarly as in~\cite{HSS}, for fixed $p$ and $\qq$,
% rewriting ! added
define a function $h: \R^{n-l}\to \R$ by 
\[
h(\xx) = \rr\cdot\xx^{t} + \boldsymbol{s}\cdot g(\xx)^{t} - a,
\]
for
$\rr, \boldsymbol{s}, a$ induced by $\qq, p$ as above. 
% rewriting ! added
Finally in case $\tilde{q}=0$, we instead let
\[
\rr=(q_1,\ldots,q_{n-l}), \qquad h(\xx)= \rr\cdot \xx^{t}  -  p - \theta, \qquad \rho=\Psi(\qq).
\]
As in~\cite{HSS} we see that in either case
\[
S(\qq,p)= \{ \xx\in K: |h(\xx)|<\rho  \}.
\]
We now identify $\nabla^2 h$ with the matrix $\Lambda$
of the theorem by the following calculation: If $\tilde{q}\ne 0$, writing $\boldsymbol{s}\cdot g(\xx)^t= s_1 g_1(\xx)+\cdots+s_lg_l(\xx)$,
for $1\le i\le n-l$, 
the $i$-th entry of the line vector 
$\nabla h(\xx)$ equals $r_i+ \boldsymbol{s}\cdot G_i(\xx)^t$ with 
\[
G_i(\xx)= (\partial g/\partial x_i)(\xx)= (\partial g_1/\partial x_i(\xx), \ldots, \partial g_l/\partial x_i(\xx)),
\]
%the $i$-th line of $\nabla g$
%evaluated at $\xx$ if $\tilde{q}\ne 0$, 
and this entry being simply
$r_i$ if otherwise $\tilde{q}=0$. In the sequel we assume $\tilde{q}\ne 0$, else
similarly the argument is analogous to~\cite{HSS}. 
So this $i$-th entry of $\nabla h(\xx)$, denote it by $\nabla h(\xx)_i$, reads
\[
\nabla h(\xx)_i = r_i + \boldsymbol{s}\cdot \left(\frac{\partial g}{\partial x_i}(\xx)\right)^t= r_i + \sum_{u=1}^{l} s_u\cdot (\partial g_u/\partial x_i(\xx)), \qquad 1\le i\le n-l.
\]
Then
\[
\nabla h(\xx)- \sum_{i=1}^{n-l} r_i \ee_i= \sum_{i=1}^{n-l} \boldsymbol{s}\cdot \left(\frac{\partial g}{\partial x_i}(\xx)\right)^t\cdot \ee_i= \sum_{i=1}^{n-l}\sum_{u=1}^{l} s_u\cdot (\partial g_u/\partial x_i(\xx))\cdot \ee_i,
\]
with $\ee_i$ the canonical base vectors in $\R^{n-l}$. Hence the 
symmetric quadratic matrix $\nabla^2 h(\xx)$ with $j$-th column
$\partial\nabla h(\xx)/\partial x_j$ has entries
\begin{equation} \label{eq:NEUE}
\nabla^2 h(\xx)_{i,j}=  \sum_{u=1}^{l} s_u\cdot \frac{\partial^2g_u}{\partial x_i\partial x_j}(\xx), \qquad 1\le i,j\le n-l,
\end{equation}
(and vanishes if $\tilde{q}=0$).
Thus indeed we may identify $\nabla^2 h$ with $\Lambda$. 
In contrast to~\cite{HSS}, the matrix $\nabla^2 h(\xx)$ now also depends
on $\qq\in\Z^n$ via its dependence on the induced rational 
unit vector $\boldsymbol{s}$,
however by a compactness argument we will deal with this issue.

By assumption of (II), for any unit vector 
$\boldsymbol{s}\in\R^l$ and any $\xx\in K$,
the determinant of $\nabla^2 h(\xx)$ does not vanish.
For simplicity assume for the moment $\qq$ and thus $\boldsymbol{s}$
are fixed.
%By compactness
%of the set $\{ \boldsymbol{s}\in\R^{l}: \Vert\boldsymbol{s}\Vert=1\}$
%and the continuous dependency of $\nabla^2 h(\xx)$ on
%$\boldsymbol{s}$, 
Then as in~\cite{HSS}, for some $\varepsilon>0$, the matrices
$\nabla^2 h(\xx)$ with $\xx\in K$ 
belong to a compact, convex set of matrices (identified with $\R^{(n-l)\times (n-l)}$)
with determinants at least $\varepsilon$. Now
the same argument as in~\cite{HSS} based on the mean value inequality gives the analogue of~\cite[Claim~2.3]{HSS}, which reads:

$\mathbf{Claim}$: If the norm of $\qq\in\Z^n$ is large enough, then
either there exists $\vv\in\R^{n-l}$ such that
\[
\Vert \nabla h(\xx)\Vert \asymp \Vert \xx-\vv\Vert, \qquad \xx\in K,
\]
or
\[
\Vert \nabla h(\xx)\Vert \asymp \Vert(\rr,\boldsymbol{s})\Vert
\asymp \Vert (\rr,1)\Vert, \qquad \xx\in K.
\]
Either way, we have $\Vert \nabla h(\xx)\Vert\ll 1$ for all $\xx\in K$.

We remark that the last claim is obvious by $\Vert \boldsymbol{s}\Vert=1$ and
\eqref{eq:NEUE}.
It should be pointed out that following the proof in~\cite{HSS}, the implied constants in the lemma will still depend on
$\qq$ in the form of the dependence on $\boldsymbol{s}$.
However, by the continuous dependence of $\nabla^2 h(\xx)$ on $\boldsymbol{s}$ and the compactness of the set $\{ \boldsymbol{s}\in\R^{l}: \Vert\boldsymbol{s}\Vert=1\}$, it is
easily seen that we can find uniform constants. 

The remainder of the proof of Theorem~\ref{thm1} works analogously
as in~\cite{HSS}, where we consider two cases 
and we replace $n$ by $n-l+1$ consistently and omit $\nabla \theta(\xx)$ and $\nabla^2 \theta(\xx)$, as we consider $\theta$ constant. It is worth noticing that in Case 1 we apply the analogue of~\cite[Lemma~2.4]{HSS}, which we again want to state explicitly for convenience:

\begin{lemma}[Hussain, Schleischitz, Simmons] \label{lemmacover}
	Assume $n,l$  are positive integers satisfying $n>l+1$.
	%\comjohannes{ i think we need this not n greater two times l} 
	Let $\phi: U \subset \R^{n-l}\to \R$ be a $C^2$ function. Fix $\alpha>0$, $\delta>0$, and $\xx\in U$ such that $B_{n-l}(\xx,\alpha) \subset U$.
	There exists a constant $C > 0$ depending only on $n$ such that if
	\begin{equation}\label{ine}
	\|\nabla \phi(\xx)\| \geq C \alpha \sup_{\zz\in U} 
	\|\nabla^2 \phi(\zz)\|,
	\end{equation}
	then the set $$S(\phi,\delta) = \{\yy\in B_{n-l}(\xx,\alpha) : |\phi(\yy)| < \|\nabla \phi(\xx)\| \delta\}$$ can be covered by $\asymp (\alpha/\delta)^{n-l-1}$ balls of radius $\delta$.
\end{lemma}

Indeed, this is precisely~\cite[Lemma~2.4]{HSS}, upon replacing
$n$ by $n-l+1$. Now following the proof from~\cite{HSS} with the change of dimensions mentioned above, and 
taking into account our observation that the implied constants in the Claim above can be considered independent of $p, \qq$ (or $\boldsymbol{s}$), indeed this ultimately leads to the sufficient 
condition $s<2(n-l-1)$ in place of $s<2n-2$ from~\cite{HSS}.
This agrees with condition (I).

\providecommand{\bysame}{\leavevmode\hbox to3em{\hrulefill}\thinspace}
\providecommand{\MR}{\relax\ifhmode\unskip\space\fi MR }
\providecommand{\MRhref}[2]{%
  \href{http://www.ams.org/mathscinet-getitem?mr=#1}{#2}
}
\providecommand{\href}[2]{#2}

\end{document}